\documentclass[notitlepage,11pt,reqno]{amsart}

\usepackage{amsopn,esint,nicefrac}
\usepackage[final]{hyperref}
\usepackage[T1]{fontenc}

\usepackage[utf8]{inputenc}
\usepackage{apptools}
\AtAppendix{\counterwithin{lemma}{section}} 
\usepackage[margin=1in]{geometry}
\allowdisplaybreaks
\newcommand{\stkout}[1]{\ifmmode\text{\sout{\ensuremath{#1}}}\else\sout{#1}\fi}

 \newcommand{\grad}{\triangledown}

\newcommand{\mg}{\mathfrak{g}}

\newcommand{\Rn}{\mathbb{R}^n}

\usepackage{graphicx,enumitem,dsfont,upgreek}
 \usepackage[dvips]{epsfig}
 \usepackage[mathscr]{eucal}
\usepackage{amscd}
\usepackage{amssymb}
\usepackage{amsthm}
\usepackage{amsmath}
\usepackage{latexsym}
\usepackage{dsfont}
\usepackage{upref}
\usepackage{hyperref}

\usepackage{color}
\theoremstyle{plain}

\newtheorem{thm}{Theorem}[section]
\theoremstyle{plain}
\newtheorem{lem}[thm]{Lemma}

\newtheorem{cor}[thm]{Corollary}

\theoremstyle{definition}

\newtheorem{rem}{Remark}[section]
\newtheorem*{maintheorem*}{Main Theorem}
\newtheorem*{maincorollary*}{Main Corollary}
{%
\setcounter{enumi}{0}

\begin{enumerate}}%
{\end{enumerate} }

{%
\setcounter{enumi}{0}

\begin{enumerate}}%
{\end{enumerate} }

\newcommand{\norm}[1]{\ensuremath{\left\|#1\right\|}}

\newcommand{\cL}{\ensuremath{\mathcal{L}}}

\newcommand{\dist}{\rm dist}

\newcommand{\R}{\ensuremath{\mathbb{R}}}

\newcommand{\loc}{\mathrm{loc}}

\newcommand{\dx}{\ensuremath{\, dx}}
\newcommand{\dy}{\ensuremath{\, dy}}
\newcommand{\dz}{\ensuremath{\, dz}}
\newcommand{\dt}{\ensuremath{\, dt}}

\numberwithin{equation}{section} \allowdisplaybreaks

\usepackage{cancel,pdfsync}

\title{The Pohozaev identity for mixed local-nonlocal operators}
\begin{document}

\author{Anup Biswas}
\address{Department of mathematics, Indian Institute of Science Education and Research Pune, Dr.\
Homi Bhabha Road, Pune 411045, India}
\thanks{This research is supported by a SwarnaJayanti fellowship SB/SJF/2020-21/03}

\begin{abstract}
In this article we prove the Pohozaev identity for the semilinear Dirichlet problem of the form
$-\Delta u + a(-\Delta)^s u = f(u)$ in $\Omega$, and $u=0$ in $\Omega^c$, where $a$ is a non-negative constant and
$\Omega$ is a bounded $C^2$ domain. We also establish similar identity for systems of equations. As applications of this identity, 
we deduce a unique continuation property of eigenfunctions and also the nonexistence of nontrivial
solutions in star-shaped domains under suitable condition on $f$.
\end{abstract}


\maketitle

\section{Introduction and Results}
For $s\in (0,1)$, we consider a mixed local-nonlocal operator of the form
\begin{equation}\label{E1.1}
\cL u = -\Delta u + a (-\Delta)^s u = f(u) \quad \text{in}\;\Omega, \quad u=0\quad \text{in}\; \Omega^c,
\end{equation}
where $a\geq 0$, $\Omega$ is a bounded $C^2$ domain in $\Rn$ and $(-\Delta)^s$ is the $s$-fractional Laplacian given by
$$(-\Delta)^s u(x)=c_{n, s}\, {\rm PV}\int_{\Rn}\frac{u(x)-u(y)}{|x-y|^{n+2s}}\dy.$$
Here $c_{n,s}$ is a normalizing constant given by
$$c_{n,s}=\frac{s 2^{2s}\Gamma\left(\frac{n+2s}{2}\right)}{\pi^{n/2}\Gamma(1-s)}.$$
When $a=0$, this is the standard local case and it is enough to prescribe the boundary data on $\partial\Omega$. To state our main result we define the following semi-norms 
\begin{align*}
[u]_s&= \left(\frac{c_{n,s}}{2}\int_{\Rn}\int_{\Rn} \frac{|u(x)-u(y)|^2}{|x-y|^{n+2s}}\dx\dy\right)^{\nicefrac{1}{2}},
\\
[u]_1& = \left(\int_\Omega|\grad u(x)|^2 \dx\right)^{\nicefrac{1}{2}}.
\end{align*}
Since we are concerned with bounded solutions and under the Lipschitz regularity of $f$, classical, viscosity and weak solutions coincide (see Remark~\ref{Rem-1.2}),
we state our main result for the classical solution.

\begin{thm}\label{T-main}
Let $\Omega$ be a bounded $C^2$ domain in $\Rn$ and $u\in C^2(\Omega)\cap C(\Rn)$ be a solution to \eqref{E1.1}.
Also, let $f$ be locally Lipschitz. Then $u\in C^{1, \alpha}(\bar\Omega)$ for some $\alpha\in (0, 1)$ and
\begin{equation}\label{ET-01A}
s\, a\,[u]^2_s +  [u]^2_1 - n \mathcal{E}(u) = 
\frac{1}{2}\int_{\partial \Omega} \left(\frac{\partial u}{\partial \nu}\right)^2 (x\cdot \nu(x)) dS,
\end{equation}
where $f(u)=F'(u)$, $\nu$ is the unit outward normal on $\partial\Omega$ and 
$$\mathcal{E}(u)= \frac{a}{2}[u]^2_{s} + \frac{1}{2}[u]^2_{1}- \int_{\Omega} F(u) \dx.$$
\end{thm}

\begin{rem}
Since $u\in C^{0,1}(\Rn)$ and $u=0$ in $\Omega^c$, we get $u\in H^s(\Rn)$. Therefore, the following 
identity holds
$$\int_\Omega u (-\Delta)^s u \dx = \int_{\Rn} (-\Delta)^{\nicefrac{s}{2}}u (-\Delta)^{\nicefrac{s}{2}}u \dx=[u]^2_s.$$
Hence \eqref{ET-01A} can be written in the following equivalent form 
$$ (s-1)\, a\,[u]^2_s + \frac{(2-n)}{2}\int_\Omega u f(u)\dx + n \int_{\Omega} F(u) \dx=\frac{1}{2}\int_{\partial \Omega} \left(\frac{\partial u}{\partial \nu}\right)^2 (x\cdot \nu(x)) dS.$$
\end{rem}

\begin{rem}\label{Rem-1.2}
Theorem~\ref{T-main} holds if either $u$ is viscosity solution in $C(\Rn)$ or a weak solution in the space
$$\mathbb{X}(\Omega)\cap L^n(\Omega)\quad \text{where}\quad \mathbb{X}(\Omega):=\{v\in {H}^1(\Rn)\; :\; v=0\quad \text{a.e.\ in $\Omega^c$}\}.$$
More precisely, if $u$ is a viscosity solution, then by \cite{BMS22}, $u\in C^{0, 1}(\Rn)\cap C^1(\bar\Omega)$. Since $f$ is locally Lipschitz, we have the map $x\mapsto f(u(x))$ Lipschitz in $\Omega$. Therefore, by \cite[Theorem~5.3]{MZ21}, $u\in C^2(\Omega)$. Hence $u$ is a classical solution and Theorem~\ref{T-main} applies.

For $a=0$, any weak solution $u\in H^1_0(\Omega)$ is in $W^{2, p}(\Omega)$ \cite[Theorem~9.15]{GilTru} for any $p>1$, and therefore, $\grad u$
is H\"{o}lder continuous in $\bar\Omega$. Now from the standard theory of elliptic PDE it follows that $u\in C^{2}(\Omega)$, implying $u$ is a
classical solution.
On the other hand, for $a>0$, if $u$ is a weak solution in $\mathbb{X}(\Omega)\cap L^n(\Omega)$, by \cite[Theorem~1]{DM24}, $u\in C^{\alpha}(\Rn)$ for all $\alpha\in (0,1)$. Let $\{u_\varepsilon\}$ be a smooth sequence of mollification of $u$. Define $g_\varepsilon=f(u_\varepsilon)$.
Let $w_\varepsilon\in C^2(\Omega)$ be the unique solution to
$$\cL w_\varepsilon = g_\varepsilon \quad \text{in}\; \Omega, \quad w_\varepsilon=0\quad \text{in}\; \Omega^c.$$
By \cite{BMS22}, we get $w_\varepsilon\in C^{1,\alpha}(\bar\Omega)$ for some $\alpha\in(0, 1)$ and from Theorem~\ref{T-03} below we also obtain 
the bound 
$$|D^2 w_\varepsilon|\leq C\, {\rm dist}(x, \partial\Omega)^{\alpha-1}\quad \text{for}\; x\in\Omega,$$
for some $C>0$ and $\alpha\in (0, 1)$ that could depend on $\varepsilon$. This bound allows us to perform integration by parts, proving that $w_\varepsilon$ is also a weak solution. 
Now using the stability of weak and viscosity solutions together with the fact that $w_\varepsilon\to u$ uniformly in $\Rn$,
we conclude that $u$ is also a viscosity solution. Now we apply previous argument to obtain that $u\in C^2(\Omega)$.
\end{rem}

The proof of Theorem~\ref{T-main} also extends to the systems of equations, giving us the following Pohozaev identity.
\begin{thm}\label{T-09}
Let $\Omega$ be bounded $C^2$ domain and $u, v\in C^2(\Omega)\cap C(\Rn)$ solve
\begin{align*}
-\Delta u + a_1 (-\Delta)^{s_1} u &= F_u(u, v)\quad \text{in}\; \Omega,\quad u=0\quad \text{in}\; \Omega^c,
\\
-\Delta v + a_2 (-\Delta)^{s_2} v  &=F_v(u, v)\quad \text{in}\; \Omega,\quad v=0\quad \text{in}\; \Omega^c,
\end{align*}
where $a_1, a_2\geq 0$, $s_1, s_2\in (0,1)$, $F$ is in $C^{1,1}_{\loc}(\R^2)$  satisfying $F(0,0)=0$.
Then
\begin{align}\label{ET-09A}
& s_1 a_1\, [u]^2_{s_1} +  [u]^2_1 + s_2 a_2\, [v]^2_{s_2} +  [v]^2_1 - n \mathcal{E}(u,v) \nonumber
\\
& \quad = \frac{1}{2}\int_{\partial \Omega} \left(\frac{\partial u}{\partial \nu}\right)^2 (x\cdot \nu(x)) dS + \frac{1}{2}\int_{\partial \Omega} \left(\frac{\partial v}{\partial \nu}\right)^2 (x\cdot \nu(x)) dS,
\end{align}
where
$$\mathcal{E}(u,v)= \frac{a_1}{2}[u]^2_{s_1} + \frac{1}{2}[u]^2_1
+ \frac{a_2}{2}[v]^2_{s_2} + \frac{1}{2}[v]^2_1-\int_{\Omega} F(u,v) \dx.$$
\end{thm}

For $a=0$, \eqref{ET-01A} reduces to the classical Pohozaev identity which was first proved by Pohozaev in his seminal work \cite{Poh}. Since then, Pohozaev identity has played a 
key role in the development of  nonlinear partial differential equation, particularly in the context of nonexistence of nontrivial solutions. Recently, Ros-Oton and Serra established the Pohozaev identity for the fractional Laplace operator in  \cite{RS14}, see also \cite{RSV} for extension to 
the anisotropic integro-differential operators of order $2s$.
 Unlike the standard Laplace operator, the solution corresponding to the fractional Laplacian with a Dirichlet boundary condition behaves like $(\dist(\cdot, \partial\Omega))^s$
near the boundary (see \cite{RS14a}). Consequently, the term $\frac{\partial u}{\partial\nu}$ on the right-hand side of \eqref{ET-01A} gets replaced by 
$\frac{ u}{(\dist(\cdot, \partial\Omega))^s}$ in the case of fractional Laplacian.
Very recently, there has been significant interest in mixed local-nonlocal operators to better understand the interplay between local and nonlocal components \cite{AC21,BDVV,BDVV2,BDVV3,BDVV4,BMS22,BM24,DM24,GK22,MS23,MZ21}.
For the problem \eqref{E1.1}, it has been shown that $u\in C^{1}(\bar\Omega)$ and $\frac{ u}{\dist(\cdot, \partial\Omega)}$ is H\"{o}lder continuous in $\bar\Omega$ \cite{BDVV3,BMS22}. Similar result 
has also been established for a class of nonlinear Pucci type integro-differential operators \cite{MS23}. Let us also mention the recent work \cite{GCAG} which establishes Pohozaev-type identity in the whole space $\Rn$ for quasilinear local-nonlocal operators.

It is important to note that the non-existence of solutions for \eqref{ET-01A} when $f(u)$ is super-critical can be analyzed without the  Pohozaev identity. More precisely,  Ros-Oton \& Serra 
demonstrate in \cite{RS15} that if $f$ satisfies certain variational inequality, which is in the spirit of Pucci \& Serrin \cite{PS86}, then mixed local-nonlocal operator can not possess a non-trivial solution. To comment on the proof of
Theorem~\ref{T-main}, we broadly follow the approach of \cite{RS14}. Note that since  the terms
 $\Delta u$ and $(-\Delta)^s u$ are not necessarily bounded in $\bar\Omega$, we can not apply \cite[Proposition~1.6]{RS14} or a straightforward integration by parts. In Theorem~\ref{T-03} we 
 establish the behaviour of these  terms near the boundary which allow us to adapt the proof of \cite{RS14}. Also, to incorporate the
approach of \cite{RS14}, we must estimate the regularity of $\mg=(-\Delta)^{\nicefrac{s}{2}} u$ in $\Rn$ (see Lemma~\ref{L05}). Since $u$ is Lipschitz in $\Rn$, thanks to \cite{BMS22}, $\mg$ exhibits a better 
regularity than the corresponding term in \cite{RS14}. This is why we do not observe any contribution from the nonlocal operator in the boundary integration.

Though there are several well known applications of Pohozaev identity,  we restrict our attention to some applications of more obvious interest and are not covered by \cite{RS15}.
The first one is a unique continuation property which can be traced backed to \cite{Poh70,Uh76} in context of elliptic operators and can also be found in \cite{RS15a} for the fractional Laplacian operator.
\begin{cor}\label{C1.3}
Let $\Omega$ be a bounded $C^2$ domain and $\phi\in C^2(\Omega)\cap C(\Rn)$ solve
\begin{align*}
\cL \phi &= \lambda\phi \quad \text{in}\; \Omega,
\\
\phi & = 0 \quad \text{in}\; \Omega^c,
\end{align*}
for some $\lambda\in \R$. Then
$$\frac{\partial \phi}{\partial \nu}=0\quad \text{on}\;\; \partial\Omega\; \Rightarrow \phi\equiv 0\quad \text{in}\; \Rn.$$
\end{cor}

Next one is an application of Theorem~\ref{T-09}. For $p$-Laplacian operators similar result can be found in \cite{BM03}.
\begin{cor}\label{C1.4}
Let $\Omega$ be a bounded, star-shaped $C^2$ domain  and $n\geq 3$. Let $u, v\in C^2(\Omega)\cap C(\Rn)$ solve the following system of equations.
\begin{equation}\label{ET1.2A0}
\begin{split}
-\Delta u + a_1 (-\Delta)^{s_1} u & = \lambda_1 |u|^{p-2}u + \delta \alpha |u|^{\alpha-2} u |v|^\beta\quad \text{in}\; \Omega,
\\
-\Delta v + a_2 (-\Delta)^{s_2} v & = \lambda_2 |v|^{q-2}v + \delta \beta |u|^{\alpha}  |v|^{\beta-2} v \quad \text{in}\; \Omega,
\\
u=0, v&=0\quad \text{in}\; \Omega^c,
\end{split}
\end{equation}
where $\alpha, \beta, p, q>1$ and $\lambda_1, \lambda_2, \delta\in\R$. Assume that, for $2^*=\frac{2n}{n-2}$, 
\begin{equation}\label{ET1.2A}
\delta[(\alpha+\beta)-2^*]\geq 0, \quad \lambda_1(p-2^*)\geq 0, \quad \text{and}\quad  \lambda_2(q-2^*)\geq 0.
\end{equation}
If the last two inequalities in \eqref{ET1.2A} are strict, then $u\equiv 0\equiv v$ in $\Rn$. Furthermore, under \eqref{ET1.2A}, 
neither of $u, v$ can be positive in $\Omega$.
\end{cor}

Next application considers Brezis-Nirenberg type problem given by
\begin{equation}\label{BN}
\begin{split}
\cL u & = |u|^{2^*-2} u + \lambda u \quad \text{in}\; \Omega,
\\
u&=0 \quad \text{in}\; \Omega^c.
\end{split}
\end{equation}
In \cite{BDVV4} the above problems are considered and the existence, nonexistence of solutions have been studied. In particular, 
it is shown that for $\lambda\leq 0$ the above problem does not have any non-trivial solution in a bounded star-shaped domain. On the other hand, there exists $\lambda^*\in [\lambda_{1,s}, \lambda_1)$ such that 
\eqref{BN} admits a positive solution for all $\lambda\in (\lambda^*, \lambda_1)$. Here $\lambda_{1,s}$ and $\lambda_1$ denote the Dirichlet principal eigenvalue of $(-\Delta)^s$ and $-\Delta+(-\Delta)^s$ in $\Omega$, respectively. 
Recall that
\begin{align}
\lambda_{1,s}&=\inf\left\{[u]^2_{\dot{H}^s}\; :\; \norm{u}_{L^2(\Rn)} =1\quad \text{and}\quad u=0\quad \text{in}\; \Omega^c\right\},\label{E1.7}
\\
\lambda_{1}&=\inf\left\{[u]^2_1+[u]^2_{\dot{H}^s} \; :\; \norm{u}_{L^2(\Omega)} =1\right\}.\nonumber
\end{align}
We show that the nonlocal term also contributes in the nonexistence theorem in the following sense.
\begin{cor}\label{C1.5}
Let $\Omega$ be a bounded, star-shaped $C^2$ domain.
For $\lambda< a(1-s)\lambda_{1,s}$, \eqref{BN} does not admit any non-trivial solution in $C^2(\Omega)\cap C(\Rn)$, and for $\lambda\leq a(1-s)\lambda_{1,s}$ it does not have a positive solution in
$C^2(\Omega)\cap C(\Rn)$.
\end{cor}

The remaining part of this article is organized as follows. In the next section (Section~\ref{S-poh}) we prove our essential estimates and then establish Theorem~\ref{T-main} in a strictly star-shaped domain.
In Section~\ref{S-proof} we provide the proofs of our main results stated above.

\section{Pohozaev identity in a strictly star-shaped domain}\label{S-poh}
The aim of this section is to prove Theorem~\ref{T-06} which, in turn, proves Pohozaev identity in a strictly star-shaped domain (see Corollary~\ref{Cor-7}).
We start with the following notation, which will be used in this article. Given a set $D\subset\Rn$, we define the $C^{0,1}$ norm as
$$
\norm{u}_{C^{0,1}(D)} =\sup_{D}|u| + \sup_{x\neq y, x, y\in D} \frac{|u(x)-u(y)|}{|x-y|}.
$$
Note that the functions in $C^{0,1}(D)$ are basically bounded Lipschitz functions on $D$.
Now we recall the following result from \cite[Theorem~1.1 and 1.3]{BMS22}.
\begin{thm}\label{T-BMS}
Let $g$ be a continuous function and $u$ be a viscosity solution to 
$$\cL u = g \quad \text{in}\; \Omega, \quad \text{and}\quad u=0\quad \text{in} \; \Omega^c.$$
Then, given $A_0>0$, $u\in C^{1, \alpha}(\bar\Omega)\cap C^{0, 1}(\Rn)$ for some $\alpha\in (0,1)$, for all $a\in [0, A_0]$, and
$$\norm{u}_{C^{1,\alpha}(\bar\Omega)} + \norm{u}_{C^{0,1}(\Rn)}\,\leq C \norm{g}_{L^\infty(\Omega)},$$
for some constant $C$, not dependent on $g$ and $u$.
\end{thm}
From the above result, we see that $\grad u$ extends continuously to the boundary of $\Omega$. In our next result, we find a bound on 
$|D^2 u|$ and $|(-\Delta)^s u|$ near $\partial\Omega$. Such bounds will be useful to justify some integration-by-parts formula that we use and also to define integration of
$\Delta u$, $(-\Delta)^s u$ over $\Omega$.
\begin{thm}\label{T-03}
Let $\alpha$ be same as in Theorem~\ref{T-BMS}. Then there exists a constant $C$ such that for any solution $u\in C^2(\Omega)\cap C(\Rn)$ to
$$\cL u = g \quad \text{in}\; \Omega, \quad \text{and}\quad u=0\quad \text{in} \; \Omega^c,$$
with $\norm{g}_{C^{0,1}(\Omega)}\leq 1$, we have
\begin{align}
|D^2 u(x)| &\leq C ({\dist} (x, \partial \Omega))^{\alpha-1}, \label{ET2.2A}
\\
|(-\Delta)^s u(x)|& \leq C \left(1+1_{s\in[0,\frac{1}{2})}+1_{\{s=\frac{1}{2}\}} |\log{\dist}(x, \partial \Omega)|+ 
1_{\{s>\frac{1}{2}\}}({\dist}(x, \partial \Omega))^{1-2s}\right), \label{ET2.2B}
\end{align}
for all $x\in \Omega$.
\end{thm}

\begin{proof}

From Theorem~\ref{T-BMS} we know that $u\in C^{1, \alpha}(\bar\Omega)$, and therefore, we get
$\grad u\in C^\alpha(\bar\Omega)$. For notational economy, we denote $\delta(x)={\dist}(x, \partial\Omega)$.
Fix a point $x_0\in \Omega$
and $4r=\delta(x_0)$. Let us now define
$$w_h(x)=\frac{u(x+eh)-u(x)}{h}-\grad u(x_0)\cdot e,$$
where $e$ is any unit vector.
Then
$$L w_h=\frac{g(x+eh)-g(x)}{h}:=G_h(x)\quad \text{in}\; B_{2r}(x_0).$$
Letting
$\xi(z)= w_h(r z+x_0)$ for $z\in B_2(0)$, we then have
$$-\Delta \xi + r^{2-2s}a\, (-\Delta)^s \xi =  r^2G_h(zr+x_0)
\quad \text{in}\; B_2(0).$$
Consider a smooth cut-off function $\chi$ satisfying $\chi=1$ in $B_1$
and $\chi=0$ in $B^c_{3/2}$. From the above equation it then follows that, for $\varphi=\xi\chi$,
$$-\Delta\varphi + r^{2-2s}a\, (-\Delta)^s \varphi = r^2 G_h(zr+x_0)
-r^{2-2s}(-\Delta)^s((1-\chi)\xi)
\quad \text{in}\; B_1(0).$$
Applying \cite[Theorem~4.1]{MZ21} we find a constant $\kappa$, independent of $r$, satisfying
$$\norm{\xi}_{C^1(B_{1/2})}\leq \kappa \left(r^2\norm{G_h(zr+x_0)}_{L^\infty(B_1)} + r^{2-2s}\norm{(-\Delta)^s((1-\chi)\xi)}_{L^\infty(B_1)} + \norm{\varphi}_{L^\infty(\Rn)}\right).
$$
By our assumption of $g$, we have 
$$\norm{G_h(zr+x_0)}_{L^\infty(B_1)}\leq 1. $$
Since $u$ is in $C^{0,1}(\Rn)$ by Theorem~\ref{T-BMS}, we have
$w_h$ is globally bounded, independent of
$r, h$ and $x_0$. Hence $|(-\Delta)^s((1-\chi)\xi)|$
is bounded in $B_1$ and the bound does not depend on $r, h$ and $x_0$. Again, $\varphi$ is supported in $B_2$ and for $z\in B_2$ we have 
$$|\varphi(z)|\leq |\xi(z)|\leq |w_h(rz+x_0)|\lesssim r^\alpha,$$
using the fact $u\in C^{1, \alpha}(\bar\Omega)$.
Therefore, computing the derivative of $\xi$ at $z=0$ it follows that
$$\frac{|\partial_{x_j} u(x_0+ he)- \partial_{x_j} u(x_0)|}{h}
\leq \kappa_2 r^{\alpha-1}\quad \text{for all}\; h\in (0,r), j=1,..,n$$
for some constant $\kappa_2$. Letting $h\to 0$ we obtain the desired bound of $|D^2 u(x_0)|$. This proves \eqref{ET2.2A}.

We are now left to show the bound of $|(-\Delta)^s u(x)|$. Fix $x\in\Omega$ and let $r=\delta(x)/4$. Then 
using \eqref{ET2.2A} and the Lipschitz bound of $u$ in Theorem~\ref{T-BMS} we obtain
\begin{align*}
&|(-\Delta)^s u(x)|
\\
&\leq c_{n,s}\int_{B_r(0)} \frac{|u(x+y)+u(x-y)-2u(x)|}{|y|^{n+2s}}\dy + c_{n,s}\int_{B^c_r(0)} \frac{|u(x+y)+u(x-y)-2u(x)|}{|y|^{n+2s}}\dy
\\
&\leq C r^{\alpha-1} \int_{B_r(0)} |y|^{2-2s-n}\dy + C \int_{B_1(0)\setminus B^c_r(0)} |y|^{1-n-2s}\dy
+ C \int_{|y|\geq 1} |y|^{-n-2s} \dy
\\
&\leq C(r^{1-2s +\alpha} + 1+1_{\{s=\frac{1}{2}\}} |\log r| + 1_{\{s>\frac{1}{2}  \}}r^{1-2s} ).
\end{align*}
This gives us \eqref{ET2.2B}, completing the proof.
\end{proof}

We also need the next key lemma to compute the contribution from the nonlocal operator at the boundary.

\begin{lem}\label{L05}
Let $\mathfrak{g}(x)=(-\Delta)^{s/2}u(x)$ where $u\in C^{0,1}(\Rn)\cap L^\infty(\Rn)\cap C^2(\Omega)$ .
Also, assume that for some constant $C$ we have
\begin{equation}\label{EL05A0}
|D^2 u(x)|\leq C ({\rm dist}(x, \partial\Omega))^{-1}
\quad \text{and}\quad  |u(x)|\leq C\, {\rm dist}(x, \partial\Omega)\quad \text{for}\;\; x\in\Omega.
\end{equation}
 Then the following hold:
\begin{itemize}
\item[(i)] $\mg\in C^{1-s}(\Rn)$
\item[(ii)] for $\beta\in [1-s, 2-s]$,
 we have
$$ \norm{\mg}_{C^\beta(\breve\Omega_\varrho)}
\leq C \varrho^{1-s-\beta}\quad \text{for}\; \varrho\in (0,1),$$
for some constant $C$, where
$$\breve\Omega_\varrho=\{x\in\Rn\; :\; {\rm dist}(x, \partial \Omega)\geq \varrho\}.$$
\end{itemize}
\end{lem}

\begin{proof}
We start with (i). Since $u$ is Lipschitz, $\mg(x)$ is bounded. More precisely,
\begin{align*}
|\mg(x)|&\leq c_{n,\frac{s}{2}}\int_{\Rn} \frac{|u(x)-u(x+z)|}{|z|^{n+s}} \dz
\\
&= c_{n,\frac{s}{2}} \int_{|z|\leq 1} \frac{|u(x)-u(x+z)|}{|z|^{n+s}} \dz
+ c_{n,\frac{s}{2}} \int_{|z|> 1} \frac{|u(x)-u(x+z)|}{|z|^{n+s}} \dz
\\
&\leq c_{n,\frac{s}{2}} \norm{u}_{C^{0,1}} \left[\int_{|z|\leq 1} |z|^{1-n-s} \dz
+ \int_{|z|> 1} \frac{1}{|z|^{n+s}}ds\right]
\\
&\leq C \norm{u}_{C^{0,1}}
\end{align*}
for all $x\in\Rn$. Therefore, for $|x-y|\geq 1$, it follows that
$$\frac{|\mg(x)-\mg(y)|}{|x-y|^{1-s}}\leq 2\norm{\mg}_{L^\infty(\Rn)}.$$
Now let $|x-y|\leq 1$. Then
\begin{align*}
|\mg(x)-\mg(y)|&\leq c_{n,\frac{s}{2}}\int_{|z|\leq |x-y|}\frac{|u(x)-u(x+z)-u(y)+u(y+z)|}{|z|^{n+s}} \dz
\\
&\qquad + c_{n,\frac{s}{2}} \int_{|z|> |x-y|}\frac{|u(x)-u(x+z)-u(y)+u(y+z)|}{|z|^{n+s}} \dz
\\
&\leq
2c_{n,\frac{s}{2}}\norm{u}_{C^{0,1}(\Rn)} \int_{|z|\leq |x-y|}|z|^{1-n-s} d{z}
+ 2c_{n,\frac{s}{2}} \norm{u}_{C^{0,1}(\Rn)}|x-y| \int_{|z|> |x-y|}\frac{1}{|z|^{n+s}}ds
\\
&\leq C \norm{u}_{C^{0,1}(\Rn)}|x-y|^{1-s}.
\end{align*}
This completes the proof of (i).

Next we consider (ii).
Consider $x\in \Omega^c$ with ${\rm dist}(x, \partial \Omega)\geq \varrho$. It is easy to see that
$$\grad \mg(x)=c_{n,\frac{s}{2}}\int_{\Rn}\frac{-\grad u(x+z)}{|z|^{n+s}}\dy
=c_{n,\frac{s}{2}}\int_{|z|\geq \varrho/2}\frac{-\grad u(x+z)}{|z|^{n+s}} \dy.$$
Since $u$ is Lipschitz,
$$\norm{\grad \mg}_{L^\infty(\Omega^c\cap\breve\Omega_\varrho)}
\lesssim \varrho^{-s}.$$
Now let $x\in\Omega\cap\breve\Omega_\varrho$. For $|h|\in (0, \varrho/4)$ and a unit vector $e$, we compute
$$\frac{\mg(x+eh)-\mg(x)}{h}= c_{n,\frac{s}{2}}\int_{\Rn} \left[\frac{u(x+eh)-u(x)}{h}- \frac{u(x+eh+z)-u(x+z)}{h} \right]\frac{1}{|z|^{n+s}}\dz.$$
Since $u\in C^{0,1}(\Rn)\cap C^2(\Omega)$, letting $h\to 0$ we obtain
$$\grad \mg(x)\cdot e= c_{n,\frac{s}{2}} \int_{\Rn} \frac{\grad u(x)\cdot e- \grad u(x+z)\cdot e}{|z|^{n+s}} \dz.$$
Again, since $z\mapsto \grad u(z)$ is continuous in 
$\Rn\setminus \partial \Omega$, we have $\grad\mg(\cdot)\cdot e$ is continuous at $x$, implying $\grad \mg(x)$ exists and 
$$\grad \mg(x)= c_{n,\frac{s}{2}} \int_{\Rn} \frac{\grad u(x)-\grad u(x+z)}{|z|^{n+s}} \dz.$$
Using the upper bound of $D^2 u$ in \eqref{EL05A0} and we note that
\begin{align*}
| \mg_{x_i}(x)|&=c_{n,\frac{s}{2}} \left|\int_{\Rn} \frac{u_{x_i}(x)-u_{x_i}(x+z)}{|z|^{n+s}} \dz \right |
\\
&\leq \kappa_1 \varrho^{-1}\int_{|z|\leq \varrho/2} |z|^{1-n-s} \dz
+ \kappa_1 \varrho^{-s}
\\
&\leq \kappa_2 \varrho^{-s},
\end{align*}
for some constants $\kappa_1, \kappa_2$. Thus,
we get
\begin{equation}\label{EL05A}
\norm{\grad \mg}_{L^\infty(\breve\Omega_\varrho)}\leq C \varrho^{-s}\quad \text{for}\; \varrho\in (0,1).
\end{equation}
Let $\beta\in [1-s, 1]$ and $x, y\in\breve\Omega_\varrho$. Then for $|x-y|\leq \varrho/4$ , we use \eqref{EL05A} to obtain
$$\frac{|\mg(x)-\mg(y)|}{|x-y|^\beta}\leq C \varrho^{-s}|x-y|^{1-\beta}\leq 4^{\beta-1}C\varrho^{1-s-\beta}$$
and for $|x-y|>\varrho/4$, we use (i) to see that
$$\frac{\mg(x)-\mg(y)}{|x-y|^\beta}\leq \norm{\mg}_{C^{1-s}(\Rn)} |x-y|^{1-s-\beta}\leq \norm{\mg}_{C^{1-s}(\Rn)} 4^{\beta+s-1} \varrho^{1-s-\beta}.$$
This proves (ii) when $\beta\in [1-s, 1]$.

Set $\beta=1+\beta'$ where $\beta'\in (0, 1-s]$. Let $x, y\in \breve\Omega_\varrho$.
If $|x-y|\geq \varrho/4$, from \eqref{EL05A} we have
$$\frac{|\grad \mg(x)-\grad \mg (y)|}{|x-y|^{\beta'}}\leq C \varrho^{-s-\beta'}= \varrho^{1-s-\beta}.$$
So we consider $|x-y|\leq \varrho/4$. If $x\in \Omega^c\cap\breve\Omega_\varrho$,
we must have $y\in \Omega^c\cap\breve\Omega_\varrho$. Then
\begin{align*}
\mg_{x_i}(x)-\mg_{x_i}(y)= c_{n, \frac{s}{2}}\int_{|z-x|\geq \frac{\varrho}{4}}u_{x_i}(z) \frac{1}{|z-x|^{n+s}}\dz - c_{n, \frac{s}{2}}\int_{|z-y|\geq \frac{\varrho}{4}}u_{x_i}(z) \frac{1}{|z-y|^{n+s}}\dz. 
\end{align*}
Since $\grad u=0$ in $\bar\Omega^c$, the above can be written as 
\begin{align*}
\mg_{x_i}(x)-\mg_{x_i}(y)= c_{n, \frac{s}{2}}\int_{|z-x|\geq \frac{\varrho}{2}}u_{x_i}(z) \left[\frac{1}{|z-x|^{n+s}}\dz - \frac{1}{|z-y|^{n+s}}\right]\dz. 
\end{align*}
For $|z|\geq \varrho/2$ and $|\tilde z|\leq \varrho/4$, we note that
$$|{|z|^{-n-s}}-{|z+\tilde z|^{-n-s}}|
= {|z|^{-n-s}}|1-\bigl|\frac{z}{|z|}-\frac{\tilde z}{|z|}\bigr|^{-n-s}|
\leq\kappa_3 \frac{1}{|z|^{n+s}}\frac{|\tilde z|}{|z|},
$$
for some constant $\kappa_3$.
Putting it in the above calculation and using the global bound of $|\grad u|$, we obtain
$$|\mg_{x_i}(x)-\mg_{x_i}(y)|\leq \kappa_4 |x-y| \int_{|z-x|\geq \frac{\varrho}{2}}\frac{1}{|z-x|^{n+1+s}}\dz \leq \kappa_5
|x-y|^{\beta'} \varrho^{1-s-\beta}.$$
Let $x, y\in \Omega\cap\breve\Omega_\varrho$ with $|x-y|\leq \varrho/4$.
Let $\theta=\frac{1}{4}{\rm dist}(x, \partial\Omega)$. Clearly, $\theta\geq \varrho/4$.
Consider a smooth cut-off function $\chi$ such that
$$\chi=1\quad \text{in}\; B_{3\theta}(x),\quad \text{and}\quad
\chi=0\quad \text{in}\; B^c_{7\theta/2}(x).$$
Note that
$$|\grad\chi|\leq \kappa \theta^{-1},\quad |D^2\chi|\leq \kappa \theta^{-2},
$$
for some $\kappa$ independent of $\varrho$.
Letting $\varphi=u\chi$ and $\xi=u(1-\varphi)$ we see that
$$u_{x_i}= \varphi_{x_i} + \xi_{x_i}.$$
Then
\begin{equation}\label{EL05B}
\mg_{x_i}(x)-\mg_{x_i}(y)=
(-\Delta)^{s/2}(\varphi_{x_i}(x)-\varphi_{x_i}(y)) + (-\Delta)^{s/2}(\xi_{x_i}(x)-\xi_{x_i}(y)).
\end{equation}
Using both the bounds in \eqref{EL05A0} we see that
$$\sup_{\Rn}|D^2\varphi|\leq \kappa_1 \varrho^{-1}$$
for some constant $\kappa_1$, independent of $\varrho$. Letting $r=|x-y|$ we compute
\begin{align*}
|(-\Delta)^{s/2}\varphi_{x_i}(x)-(-\Delta)^{s/2}\varphi_{x_i}(y)|
&\leq c_{n, \frac{s}{2}}\left|\int_{|z|\leq r}\frac{\varphi_{x_i}(z+x)
-\varphi_{x_i}(x)-\varphi_{x_i}(z+y)+\varphi_{x_i}(y)}{|z|^{n+s}}\dz \right|
\\
&\qquad + c_{n, \frac{s}{2}}\left|\int_{|z|> r}\frac{\varphi_{x_i}(z+x)
-\varphi_{x_i}(x)-\varphi_{x_i}(z+y)+\varphi_{x_i}(y)}{|z|^{n+s}}\dz \right|
\\
&\lesssim \varrho^{-1} \int_{|z|\leq r} |z|^{1-n-s}\dz
+ \varrho^{-1} |x-y| \int_{|z|>r}|z|^{-n-s}
\\
&\lesssim \varrho^{-1} r^{1-s}
\\
&\leq \kappa_4 r^{\beta'} \varrho^{-s-\beta'}= \kappa_4 r^{\beta'} \varrho^{1-s-\beta}
\end{align*}
for some constant $\kappa_4$. Now consider the last term in \eqref{EL05B}. Note that
by our choice of cut-off function
\begin{align*}
|(-\Delta)^{s/2}(\xi_{x_i}(x)-\xi_{x_i}(y))| &= c_{n, \frac{s}{2}} \left| \int_{|z|\geq \varrho/4} (\xi_{x_i}(x+z)-\xi_{x_i}(y+z))\frac{\dz}{|z|^{n+s}}\right|
\\
&= c_{n, \frac{s}{2}} \left| \int_{|z-x|\geq \varrho/2} \xi_{x_i}(z)\frac{\dz}{|z-x|^{n+s}} - \int_{|z-x|\geq \varrho/2} \xi_{x_i}(z)\frac{\dz}{|z-y|^{n+s}}\right|
\\
&\leq \kappa_5 |x-y| \varrho^{-s-1}\leq  \kappa_6 |x-y|^{\beta'} \varrho^{1-s-\beta}
\end{align*}
for some constants $\kappa_5, \kappa_6$ where the estimate follows applying the same estimate as done before 
(see the estimate for $x, y\in\Omega^c\cap\breve\Omega_\varrho$).
Combining these estimates we have the proof.
\end{proof}

In the remaining part of this section, we set $\Omega$ to be a strictly star-shaped domain.
Recall that
$\Omega$ is said to be strictly star-shaped if there exists a point $z_0\in\Omega$ such that
$$(x-z_0)\cdot\nu(x)>0 \quad \text{for all}\; x\in\partial\Omega.$$

\begin{thm}\label{T-06}
Let $\Omega$ be a bounded, strictly star-shaped  $C^2$ domain. Let $u\in C^2(\Omega)\cap C^{1}(\bar\Omega)\cap C^{0,1}(\Rn)$ satisfy the following
for $x\in\Omega$
\begin{equation}\label{ET06A}
\begin{split}
|D^2 u(x)| &\leq C ({\rm dist}(x, \partial\Omega))^{\gamma-1},
\\
|u(x)|&\leq C\, {\rm dist}(x, \partial\Omega),
\\
|(-\Delta)^s u(x)|& \leq C \left(1+1_{\{s\in[0,\frac{1}{2})\}}+1_{\{s=\frac{1}{2}\}} |\log {\rm dist}(x, \partial\Omega)|+ 
1_{\{s>\frac{1}{2}\}}({\rm dist}(x, \partial\Omega))^{1-2s}\right)
\end{split}
\end{equation}
for some constants $C>0$ and $\gamma\in (0,1)$. 
Then we have
\begin{align}
\int_{\Omega}(x\cdot\grad u)(-\Delta )^s u\dx &= \frac{2s-n}{2}\int_{\Omega}u(-\Delta)^s u \dx= \frac{2s-n}{2}[u]^2_{s},\label{ET06B}
\\
\int_{\Omega}(x\cdot\grad u)(-\Delta u)\dx&= \frac{2-n}{2}[u]^2_1-\frac{1}{2}\int_{\partial\Omega}
\left(\frac{\partial u}{\partial\nu}\right)^2(x\cdot \nu) dS.\label{ET06C}
\end{align}
\end{thm}

As a corollary to Theorem~\ref{T-06} we obtain the Pohozaev identity in a star-shaped domain.
\begin{cor}\label{Cor-7}
Let $\Omega$ be a bounded, strictly star-shaped  $C^2$ domain and $u\in C^2(\Omega)\cap C(\Rn)$ be a solution to
\begin{equation}\label{EC07A}
\cL u = f(u) \quad \text{in}\; \Omega, \quad \text{and}\quad u=0\quad \text{in} \; \Omega^c,
\end{equation}
for some locally Lipschitz function $f$. Then
\begin{equation}\label{EC07B}
s\, a\,[u]^2_s + [u]^2_1 - n \mathcal{E}(u) = 
\frac{1}{2}\int_{\partial \Omega} \left(\frac{\partial u}{\partial \nu}\right)^2 (x\cdot \nu(x)) dS,
\end{equation}
where $f(u)=F'(u)$.
\end{cor}

\begin{proof}
In view of Theorem~\ref{T-BMS} and ~\ref{T-03}, we see that \eqref{ET06A} holds. Thus the conclusion of Theorem~\ref{T-06} holds.
Hence \eqref{EC07B} follows by multiplying \eqref{EC07A} with $(x\cdot\grad u)$, integrating over $\Omega$ and applying the following relation.
$$\int_\Omega (x\cdot\grad u(x)) f(u)\dx= \int_\Omega x\cdot\grad  F(u)\dx=-n\int_\Omega F(u) \dx.$$
\end{proof}

Now we focus on the proof of Theorem~\ref{T-06}.
We follow the ideas of \cite{RS14} to prove Theorem~\ref{T-06}.
Without loss of generality, we assume that $\Omega$ is strictly star-shaped with respect to $z_0=0$. In what follows, we use the following notation: for a function $g:\Rn\to\R$
and a scalar $\lambda>0$, we define $g_\lambda(x)=g(\lambda x)$. The proof of \eqref{ET06B} is divided into following steps.
\begin{enumerate}
\item[\hypertarget{A}{(A)}] We show that
$$\int_{\Omega}|(-\Delta)^s u(x)|\dx<\infty\quad \text{and}\quad 
\int_\Omega (x\cdot \grad u(x))(-\Delta)^s u(x)=\frac{d}{d\lambda}\Big|_{\lambda=1+}\int_\Omega u_\lambda(x)(-\Delta)^s u(x)\dx.$$

\item[\hypertarget{B}{(B)}] We next show that for $\lambda>0$
$$ \int_\Omega u_\lambda(-\Delta)^s u(x) dx=\lambda^{\frac{2s-n}{2}}\int_{\Rn}\mg_{\sqrt{\lambda}}(x)\mg_{\frac{1}{\sqrt{\lambda}}}(x)\dx,$$
  where $\mg(x):=(-\Delta)^{\frac{s}{2}}u(x)$.
  
\item[\hypertarget{C}{(C)}] In the third step, we show that
$$\frac{d}{d\lambda}\bigg|_{\lambda=1+}\int_{\Rn}\mg_{\lambda}(x)\mg_{\frac{1}{\lambda}}(x) \dx=0.$$
\end{enumerate}

Let us first complete the proof of \eqref{ET06B} assuming the above steps.
\begin{proof}[Proof of \eqref{ET06B}]
Since
$$\frac{d}{d\lambda}\bigg|_{\lambda=1+}\int_{\Rn}\mg_{\sqrt{\lambda}}(x)\mg_{\frac{1}{\sqrt{\lambda}}}(x)\dx=\frac{1}{2}\frac{d}{d\lambda}\bigg|_{\lambda=1+}\int_{\Rn}\mg_{\lambda}(x)\mg_{\frac{1}{\lambda}}(x)\dx,$$
from the product rule of the derivative and steps \hyperlink{A}{(A)}-\hyperlink{C}{(C)}, it follows that
$$\int_\Omega (x\cdot \grad u(x))(-\Delta)^s u(x)= \frac{2s-n}{2} \int_{\Rn}  (-\Delta)^{\frac{s}{2}}u(x) (-\Delta)^{\frac{s}{2}}u(x)\dx=\frac{2s-n}{2}\int_\Omega u 
(-\Delta)^{s}u(x)\dx,$$
where the last inequality follows from \hyperlink{B}{(B)} taking $\lambda=1$. Using the dominated convergence theorem, we also see that
\begin{align*}
\int_\Omega u (-\Delta)^{s}u(x)\dx &= \int_{\Rn} u (-\Delta)^{s}u(x)\dx
\\
&=\lim_{\varepsilon\to 0} c_{n, s}\int_{\Rn} u(x) \dx \int_{|y-x|\geq \varepsilon} \frac{u(x)-u(y)}{|y-x|^{n+2s}}\dy
\\
&=\lim_{\varepsilon\to 0} \frac{c_{n, s}}{2}\int_{\Rn}  \int_{|y-x|\geq \varepsilon} \frac{(u(y)-u(x))^2}{|y-x|^{n+2s}}\dy\dx
\\
&=[u]^2_{s}.
\end{align*}
This proves \eqref{ET06B}.
\end{proof}

Now we prove \hyperlink{A}{(A)}-\hyperlink{C}{(C)}. In fact, the proofs of \hyperlink{A}{(A)}-\hyperlink{B}{(B)} are somewhat similar to the one
appearing in \cite{RS14} apart from the fact that $|(-\Delta)^su|$ is not in $L^\infty(\Omega)$ for us. This is compensated by the decay of
$u$ at the boundary. The proof of \hyperlink{C}{(C)} differs from \cite{RS14}. Since $u(x)$ behaves like ${\rm dist}(x, \partial\Omega)$ near the boundary, we do not get any contribution of boundary integration from the nonlocal part. Lemma~\ref{L05} plays a key role in proving \hyperlink{C}{(C)}.

\begin{proof}[Proof of \hyperlink{A}{(A)}]
From the bound of $|(-\Delta)^su|$ in \eqref{ET06A} and co-area formula it follows that
$$\int_{\Omega}|(-\Delta)^s u(x)|\dx<\infty.$$

Define $g=(-\Delta)^s u$. Since $u\in L^\infty (\Rn)$, $u_\lambda g\in L^{1}(\Omega)$ for all $\lambda>0$.
Thus,  making the change of variables $y=\lambda x$ and using that ${\rm supp} (u_\lambda)=\frac{1}{\lambda}\Omega\subset\Omega$ for $\lambda>1$, we obtain
\begin{align*}
\frac{d}{d\lambda}\bigg|_{\lambda=1+} \int_{\Omega} u_\lambda(x)g(x)\dx&= \lim_{\lambda\downarrow 1}\int_{\Omega}\frac{u(\lambda x)-u(x)}{\lambda-1}g(x)\dx
\\
&=\lim_{\lambda\downarrow 1}\lambda^{-n}\int_{\lambda\Omega}\frac{u(y)-u({y}/{\lambda})}{\lambda-1}g({y}/{\lambda})\dy
\\
&=\lim_{\lambda\downarrow 1}\int_{\Omega}\frac{u(y)-u({y}/{\lambda})}{\lambda-1}g({y}/{\lambda})\dy + 
\lim_{\lambda\downarrow 1}\int_{\lambda\Omega\setminus\Omega}\frac{-u({y}/{\lambda})}{\lambda-1}g({y}/{\lambda})\dy.
\end{align*}
Since $u\in C^1(\bar\Omega)$, using dominated convergence theorem together with the
bound of $g$ in \eqref{ET06A}, it is easily seen that
 $$\lim_{\lambda\downarrow 1}\int_{\Omega}\frac{u(y)-u({y}/{\lambda})}{\lambda-1}g({y}/{\lambda})dy=\int_{\Omega}(y\cdot\grad u)g(y)\dy.$$
 On the other hand, the bounds of $u$ and $g$ in \eqref{ET06A} confirms that
 $$|u(x)g(x)|\to 0 \quad \text{uniformly, as}\; x\; \text{approaches to}\; \partial\Omega.$$
 Therefore,
 $$\lim_{\lambda\downarrow 1}\left|\int_{\lambda\Omega\setminus\Omega}\frac{-u({y}/{\lambda})}{\lambda-1}g\big({y}/{\lambda}\big)\dy\right|
 \leq \frac{1}{\lambda-1}\|u_{\frac{1}{\lambda}}g_{\frac{1}{\lambda}}\|_{L^\infty(\lambda\Omega\setminus\Omega)}|\lambda\Omega-\Omega|\to 0,$$
 as $|\lambda\Omega-\Omega|\leq C(\lambda-1)|\Omega|$. This proves \hyperlink{A}{(A)}.
\end{proof}

Next, we prove \hyperlink{B}{(B)}.

\begin{proof}[Proof of \hyperlink{B}{(B)}]
Since $u\in C^{0,1}(\Rn)$ and $u\equiv 0$ in $\Omega^c$, $u\in H^s(\Rn)$. Therefore,
using integration by parts formula
 \begin{align*}
 \int_\Omega u_\lambda(x)(-\Delta)^s u(x) \dx=\int_{\Rn}u_\lambda(x)(-\Delta)^su(x) \dx
 &=\int_{\Rn}(-\Delta)^{\frac{s}{2}} u_\lambda (-\Delta)^{\frac{s}{2}}u\dx
 \\
 &= \lambda^s \int_{\Rn}(-\Delta)^{\frac{s}{2}}u(\lambda x)(-\Delta)^{\frac{s}{2}}u(x)\dx
 \\
 &=\int_{\Rn} \mg_\lambda(x) \mg(x) \dx,
 \end{align*}
  where $\mg(x):=(-\Delta)^{\frac{s}{2}}u(x)$.
 Next, using the change of variable $y=\sqrt{\lambda}x$, we have
 $$ \int_\Omega u_\lambda(-\Delta)^s u\dx=\lambda^\frac{2s-n}{2}\int_{\Rn}\mg_{\sqrt{\lambda}}\mg_{\frac{1}{\sqrt{\lambda}}}\dx.$$
 Hence \hyperlink{B}{(B)} is proved. 
\end{proof}

Let us now prove \hyperlink{C}{(C)}.
\begin{proof}[Proof of \hyperlink{C}{(C)}]
Writing the integration in polar coordinate we note that 
$$\int_{\Rn}\mg_{\lambda}(x)\mg_{\frac{1}{\lambda}}(x) \dx=
\int_{\partial\Omega} (x\cdot\nu(x)) dS\int_0^\infty t^{n-1} \mg(\lambda t x) \mg(tx/\lambda) \dt.$$
Consider a smooth cut-off function $\chi[0, \infty)\to [0, 1]$ satisfying
$$\chi(t)=0\quad \text{in}\; [0, \frac{1}{4}], \quad \chi(t)=1\quad \text{in}\; [\frac{1}{2},\infty).$$
Fix $x_0\in \partial \Omega$ and define
$$\zeta(t)=t^{\frac{n-1}{2}}\chi(t)(-\Delta)^{\frac{s}{2}} u(tx_0)\quad t>0.$$
This cut-off function will remove the singularity of the map $t\mapsto t^{\frac{n-1}{2}}$ at $0$.
We claim that for some constant $C_0$, not depending on $x_0$, we have
\begin{itemize}
\item[(i)] $\norm{\zeta}_{C^{\upkappa_1}[0, \infty)}\leq C_0$ where $\upkappa_1= (1-s)\wedge s$,
\item[(ii)] For any $\beta\in [\upkappa_1, 1+\upkappa_1]$ we have
$$\norm{\zeta}_{C^\beta((0,1-\varrho)\cup(1+\varrho, 2))}\leq C_0 \varrho^{-\beta}\quad \text{for}\; \varrho\in (0,1),$$
\item[(iii)] $|h'(t)|\leq C_0 t^{-2-\upkappa_1}$ and $|h^{\prime\prime}(t)|\leq C_0 t^{-3-\upkappa_1}$ for all $t\geq 2$.
\end{itemize}
Once this claim is established, we can apply \cite[Proposition~1.11]{RS14} to see that
$$\frac{d}{d\lambda}\bigg|_{\lambda=1+} \int_0^\infty \zeta(\lambda t)\zeta(t/\lambda)=0.$$
Furthermore, this limit of derivative is uniform for all $x_0\in\partial\Omega$. Hence 
\begin{align*}
&\frac{d}{d\lambda}\bigg|_{\lambda=1+}\int_{\Rn}\mg_{\lambda}(x)\mg_{\frac{1}{\lambda}}(x) \dx
\\
&=\frac{d}{d\lambda}\bigg|_{\lambda=1+}\int_{\partial\Omega} (x\cdot\nu(x)) dS\int_0^\frac{3}{4} t^{n-1}
[ (1-\chi(\lambda t ))\mg(\lambda t x) (1-\chi (t /\lambda))\mg(tx/\lambda)
\\
&\quad + (1-\chi(\lambda t ))\mg(\lambda t x)\chi(t/\lambda)\mg(tx/\lambda) + \chi(\lambda t)\mg(\lambda tx) (1-\chi( t /\lambda))\mg( t x/\lambda)] \dt.
\end{align*}
Note that for $x\in\partial\Omega$ and $t\leq 3/4$, ${\rm dist}(tx, \partial\Omega)\geq \frac{1}{4}$. Therefore, by Lemma~\ref{L05}(ii),
$\lambda\mapsto \mg(\lambda t x)$ is differentiable at $1$, giving us
 \begin{align*}
&\frac{d}{d\lambda}\bigg|_{\lambda=1+}\int_{\Rn}\mg_{\lambda}(x)\mg_{\frac{1}{\lambda}}(x) \dx
\\
&=\int_{\partial\Omega} (x\cdot\nu(x)) dS\int_0^\frac{3}{4} t^{n-1}
\frac{d}{d\lambda}\bigg|_{\lambda=1}[ (1-\chi(\lambda t ))\mg(\lambda t x) (1-\chi (t /\lambda))\mg(tx/\lambda)
\\
&\quad + (1-\chi(\lambda t ))\mg(\lambda t x)\chi(t/\lambda)\mg(tx/\lambda) + \chi(\lambda t)\mg(\lambda tx) (1-\chi( t /\lambda))\mg( t x/\lambda)] \dt=0,
\end{align*}
proving \hyperlink{C}{(C)}. Thus to complete the proof we need to establish the claim (i)-(iii).

We start with (iii). Note that for $t\geq 2$, $x_0t\in (2\Omega)^c$. Thus
$$(-\Delta)^{\nicefrac{s}{2}}u(tx_0)=c_{n,\frac{s}{2}}\int_{\Omega} \frac{-u(y)}{|tx_0-y|^{n+s}}\dy.$$
Since $\Omega$ is strictly star-shaped, there exists $\kappa>0$, independent of $x$, satisfying $|x-y|\geq \kappa$ for 
all $y\in\Omega$ and $x\in (2\Omega)^c$. Thus, for some constant $C$ we have
$$|(-\Delta)^{\nicefrac{s}{2}}u(x)|\leq C |x|^{-n-s},\quad
|\partial_{x_i}  (-\Delta)^{\nicefrac{s}{2}}u(x)|\leq C |x|^{-n-s-1},\quad |\partial^2_{x_ix_j}  (-\Delta)^{\nicefrac{s}{2}}u(x)|\leq C |x|^{-n-s-2},$$
for all $x\in (2\Omega)^c$. Thus for any $\gamma\leq s$ we have
$$|\zeta'(t)|\leq C_0 t^{\frac{n-1}{2}-n-s-1}\leq C_0 t^{-2-\gamma}\quad \text{and}\quad 
|\zeta^{\prime\prime}(t)|\leq  C_0 t^{-3-\gamma},$$
for all $t\geq 2$. This proves (iii).

Now we prove (i). From the above estimate we know that $\zeta$ is a bounded function, uniformly in $x_0\in \partial\Omega$.
Furthermore, $\zeta$ is Lipschitz in $[2, \infty)$ from the above estimate.
Recall the function $\mg$ from Lemma~\ref{L05}. Consider $t_1, t_2\in [0, 2]$. If $|t_1-t_2|\geq 1$, then
$$\frac{|\zeta(t_1)-\zeta(t_2)|}{|t_1-t_2|^{\upkappa_1}}\leq 2\norm{\zeta}_\infty.$$
Now suppose $|t_1-t_2|< 1$. Since $t\mapsto t^{\frac{n-1}{2}}\chi(t)$ is Lipschitz in $[0, 2]$, we have
\begin{align*}
\frac{|\zeta(t_1)-\zeta(t_2)|}{|t_1-t_2|^{\upkappa_1}}\leq |t_1^{\frac{n-1}{2}}|\frac{|\mg(t_1x_0)-\mg(t_2x_0)|}{|t_1-t_2|^{\upkappa_1}}
+|\mg(t_2)| |\frac{|\chi(t_1)t^{\frac{n-1}{2}}_1-\chi(t_2)t^{\frac{n-1}{2}}_2|}{|t_1-t_2|^{\upkappa_1}}\leq C_0
\end{align*}
by Lemma~\ref{L05}(i). This proves (i). Similarly, (ii) follows from Lemma~\ref{L05}(ii).
\end{proof}

Now we complete the proof of Theorem~\ref{T-06}.
\begin{proof}[Proof of Theorem~\ref{T-06}]
Since we already proved \eqref{ET06B} above, we need to consider \eqref{ET06C}. 
Since $u\in C^2(\Omega)$, we write the following identity (which does not require $\Omega$ to be star-shaped abound $0$)
$$\Delta u (x\cdot \grad u)= {\rm div} (\grad u (x\cdot \grad u)) - |\grad u|^2 - x\cdot\grad (\frac{1}{2}|\grad u|^2)\quad
\text{in}\; \Omega.$$
Since $u\in C^1(\bar\Omega)$ and $u=0$ on $\partial\Omega$, we have $\grad u(x)=\pm|\grad u(x)|\nu(x)$ on $\partial\Omega$.
Therefore, integrating the above identity over $\Omega$ and applying integration by parts we obtain
\begin{align*}
-\int_\Omega \Delta u (x\cdot \grad u) \dx &= -\int_{\partial\Omega} (\grad u\cdot\nu(x))(x\cdot \grad u)dS
+ \int_\Omega |\grad u|^2 \dx
\\
&\qquad  - \frac{n}{2} \int_\Omega |\grad u|^2\dx + \int_{\partial\Omega} (\nu(x)\cdot x) (\frac{1}{2}|\grad u|^2) dS
\\
&= \frac{2-n}{2} \int_\Omega |\grad u|^2\dx -\frac{1}{2}\int_{\partial\Omega} \left(\frac{\partial u}{\partial\nu}\right)^2(x\cdot \nu) dS.
\end{align*}
This proves \eqref{ET06C}, completing the proof of Theorem~\ref{T-06}.
\end{proof}


\section{Proof of Theorems~\ref{T-main},~\ref{T-09} and Corollaries~\ref{C1.3},~\ref{C1.4},~\ref{C1.5}}\label{S-proof}
First, we complete the proof of Theorem~\ref{T-main} following the ideas of \cite[Proposition~1.6]{RS14}. 
\begin{proof}[Proof of Theorem~\ref{T-main}]
First consider balls $B_1, \ldots, B_m$ of radius $r$ covering $\bar\Omega$ so that the following holds: If $B_i\cap B_j\neq \emptyset$ for $i$ and $j$, then there exists a ball $B$ and a point $z_0\in \Omega\cap B$ such that
$$ (x-z_0)\cdot\nu(x)>0\quad \text{for all}\; x\in\partial\Omega\cap B.$$
Let $\{\psi_k\, :\, k=1, \ldots, m\}$ be a partition of unity with respect to the balls $B_1, \ldots, B_m$ satisfying ${\rm supp}(\psi_k)\subset B_k$.
Define $u_k=\psi_k u$. 
We have two possibilities:

\begin{itemize}
\item[(i)] $ B_i\cap {B}_j=\emptyset$. In this case
$$-\int_\Omega (x\cdot \grad u_i)\Delta u_j=0=
\frac{2-n}{2}\int_{\Omega} \grad u_i\cdot \grad u_j \dx-\frac{1}{2}\int_{\partial\Omega}
 \frac{\partial u_i}{\partial\nu} \frac{\partial u_j}{\partial\nu} (x\cdot \nu) dS.$$
Again, \cite[Lemma~5.1]{RS14} gives
\begin{align*}
&\int_{{B}_i} (x\cdot \grad u_i) (-\Delta)^s u_j\dx  +\int_{{B}_j} (x\cdot \grad u_j) (-\Delta)^s u_i\dx
\\
&\quad = \frac{2s-n}{2} \int_{{B}_i} u_i (-\Delta)^s u_j \dx + \frac{2s-n}{2} \int_{{B}_i} u_j (-\Delta)^s u_i \dx.
\end{align*}

\item[(ii)] $B_i\cap {B}_j\neq\emptyset$. Consider the ball $B$ and point $z_0$ as mentioned above. Pick a $C^2$ domain $\tilde\Omega$ satisfying
$$\{u\neq 0\}\subset \tilde\Omega \subset \Omega\cap B\quad
\text{and}\quad (x-z_0)\cdot\nu(x)>0 \quad \text{for all}\; x\in\tilde\Omega.$$
Applying Theorem~\ref{T-03} it is easily seen that $u_i, u_j$ satisfy \eqref{ET06A} with respect to the domain $\tilde\Omega$.
Thus applying Theorem~\ref{T-06} to the function $(\chi_i+\chi_j)u$ and $(\chi_i-\chi_j)u$, and then subtracting the corresponding identities (see \eqref{ET06B}-\eqref{ET06C}) we obtain
\begin{align*} 
&\int_{\Omega}(x\cdot\grad u_i)(-\Delta)^s u_j\dx + \int_{\Omega}(x\cdot\grad u_j)(-\Delta)^s u_i\dx 
\\
&\qquad= \frac{2s-n}{2}\int_{\Omega}u_i(-\Delta)^s u_j \dx + \frac{2s-n}{2}\int_{\Omega}u_j(-\Delta)^s u_i \dx,
\end{align*}
and
\begin{align*}
\int_{\Omega}(x\cdot\grad u_i)(-\Delta u_j)\dx + \int_{\Omega}(x\cdot\grad u_j)(-\Delta u_i)\dx&= (2-n)\int_{\Omega}\grad u_i \grad u_jdx
\\
 &-\int_{\partial\Omega\cap B}\left(\frac{\partial u_i}{\partial\nu} \frac{\partial u_j}{\partial\nu}\right)(x\cdot \nu) dS.
\end{align*}
\end{itemize}
Combining the above results we obtain the identities \eqref{ET06B} and \eqref{ET06C}. Hence proof of Theorem~\ref{T-main} follows (see the proof of Corollary~\ref{Cor-7}).
\end{proof}

Now we come to the proof of Theorem~\ref{T-09}.
\begin{proof}[Proof of Theorem~\ref{T-09}]
Multiply the first equation by $(x\cdot\grad u)$ and the second equation by $(x\cdot\grad v)$, integrate over $\Omega$, and then summing
them we obtain \eqref{ET-09A} with the help of following identity
$$\int_\Omega (x\cdot \grad u) F_u + (x\cdot \grad v) F_v \dx= \int_\Omega x\cdot \grad F(u,v) \dx=-n\int_\Omega F(u, v)\dx.$$
\end{proof}

Now we focus on the proofs of Corollaries~\ref{C1.3}, ~\ref{C1.4} and ~\ref{C1.5}.
\begin{proof}[Proof of Corollary~\ref{C1.3}]
From the identity \eqref{ET-01A} we obtain
\begin{equation}\label{ET1.1A}
a s [\phi]^2_{s} + [\phi]^2_1 - n \left(\frac{1}{2}[\phi]^2_{1}+\frac{a}{2}[\phi]^2_{s}-\int_\Omega F(\phi)\right) =0.
\end{equation}
Note that $F(\phi)=\frac{\lambda}{2}\phi^2$. Moreover, multiplying the equation by $\phi$ and using integration by parts we also have
$$[\phi]^2_{1} + [\phi]^2_{s}=-\lambda \int_\Omega \phi^2 \dx.$$
Plug it in \eqref{ET1.1A} gives 
$$a s[\phi]^2_{s} + [\phi]^2_1=0,$$
giving us $\phi\equiv 0$ in $\Rn$.
\end{proof}

\begin{proof}[Proof of Corollary~\ref{C1.4}]
Translating the domain we may assume $\Omega$ to be star-shaped with respect to $0$.
We rewrite \eqref{ET-09A} as follows
\begin{align}\label{ET1.2B}
&\frac{2-n}{2}\left([u]^2_1 + a_1 [u]^2_{{s_1}} \right) + \frac{2-n}{2}\left( [v]^2_1  + a_2 [v]_{{s_2}} \right) + n\int_\Omega F(u, v)\dx \nonumber
\\
&\quad = \int_{\partial \Omega} \left[\left(\frac{\partial u}{\partial \nu}\right)^2 + \left(\frac{\partial v}{\partial \nu}\right)^2\right](x\cdot \nu(x)) dS + a_1(1-s_1)[u]_{s_1} + a_2(1-s_2)[v]_{s_2}.
\end{align}
Multiplying the first equation in \eqref{ET1.2A0} by $u$ and integrating over $\Omega$ we obtain
$$[u]_1^2 + a_1 [u]^2_{s_1}= \int_{\Omega} (\lambda_1 |u|^{p} + \delta \alpha |u|^{\alpha} |v|^\beta)\dx.$$
Similarly, the second equation gives
$$[v]_1^2 + a_2 [v]^2_{s_2}= \int_{\Omega} (\lambda_2 |v|^{q} + \delta \beta |u|^{\alpha} |v|^\beta)\dx.$$
Putting them in \eqref{ET1.2B} and using the fact
$$F(u, v)=\lambda_1\frac{1}{p} |u|^p + \lambda_2\frac{1}{q}|v|^q + \delta |u|^{\alpha}|v|^\beta,$$
we get
\begin{align}\label{ET1.2C}
&\lambda_1\left(\frac{n}{p}- \frac{n-2}{2}\right) \int_\Omega |u|^p + \lambda_2 \left(\frac{n}{q}- \frac{n-2}{2}\right) \int_\Omega |v|^q + 
\delta(n -\frac{n-2}{2}(\alpha+\beta))\int_\Omega |u|^\alpha |v|^\beta\nonumber
\\
&\quad = \int_{\partial \Omega} \left[\left(\frac{\partial u}{\partial \nu}\right)^2 + \left(\frac{\partial v}{\partial \nu}\right)^2\right](x\cdot \nu(x)) dS + a_1(1-s_1)[u]_{s_1} + a_2(1-s_2)[v]_{s_2}.
\end{align}
From \eqref{ET1.2A} we see that the left hand side of \eqref{ET1.2C} is non-positive whereas the right hand side is non-negative. Hence both the sides has to be $0$, that is,
\begin{align*}
\lambda_1\left(\frac{n}{p}- \frac{n-2}{2}\right) \int_\Omega |u|^p + \lambda_2 \left(\frac{n}{q}- \frac{n-2}{2}\right) \int_\Omega |v|^q + 
\delta(n -\frac{n-2}{2}(\alpha+\beta))\int_\Omega |u|^\alpha |v|^\beta &= 0,
\\
\int_{\partial \Omega} \left[\left(\frac{\partial u}{\partial \nu}\right)^2 + \left(\frac{\partial v}{\partial \nu}\right)^2\right](x\cdot \nu(x)) dS + a_1(1-s_1)[u]_{s_1} + a_2(1-s_2)[v]_{s_2}&=0.
\end{align*}
Now if the last two inequalities of \eqref{ET1.2A} are strict, the first equation above gives $u\equiv 0\equiv v$ in $\Rn$. Furthermore, if $u>0$ in $\Omega$, 
by Hopf's lemma (cf. \cite[Theorem~2.2]{BM24}) we know that $\frac{\partial u}{\partial\nu}<0$ on $\partial\Omega$. This contradicts the second equation above. The same contradiction holds if
$v>0$ in $\Omega$. This completes the proof.
\end{proof}

\begin{proof}[Proof of Corollary~\ref{C1.5}]
As before, we assume $\Omega$ to be star-shaped with respect to $0$. Rewrite \eqref{ET-01A} as follows
$$\frac{2-n}{2}\left([u]^2_1 + a [u]^2_{s}\right) + n\int_{\Omega} F(u) \dx
= \frac{1}{2}\int_{\partial \Omega} \left(\frac{\partial u}{\partial \nu}\right)^2 (x\cdot \nu(x)) dS + a(1-s)[u]^2_{s}.$$
Observing $F(u)=\frac{1}{2^*}|u|^{2^*}+\frac{\lambda}{2}u^2$ and using the equation we obtain
\begin{equation*}
\lambda\int_\Omega u^2\dx = \frac{1}{2}\int_{\partial \Omega} \left(\frac{\partial u}{\partial \nu}\right)^2 (x\cdot \nu(x)) dS + a(1-s)[u]^2_{\dot{H}^s}.
\end{equation*}
Using \eqref{E1.7} this gives
$$\frac{1}{2}\int_{\partial \Omega} \left(\frac{\partial u}{\partial \nu}\right)^2 (x\cdot \nu(x)) dS + (a(1-s)\lambda_{1,s}-\lambda)\int_\Omega u^2\dx\leq 0.
$$
This is clearly a contradiction if $\lambda<a(1-s)\lambda_{1,s}$, unless $u\equiv 0$. Again, when $\lambda=a(1-s)\lambda_{1,s}$, 
the above inequality implies $\frac{\partial u}{\partial \nu}=0$ on $\partial\Omega$ whenever $(x\cdot \nu(x))>0$. But this is not
possible if $u>0$ in $\Omega$, due to Hopf's lemma. This completes the proof.
\end{proof}



\begin{thebibliography}{77}
\bibitem{GCAG} Gurdev Chand Anthal and Prashanta Garain.
Pohozaev-type identities for classes of quasilinear elliptic local and
nonlocal equations and systems, with applications. ArXiv: 2506.08667, 2025

\bibitem{AC21} N. Abatangelo and M. Cozzi. An elliptic boundary value problem with fractional nonlinearity, \emph{SIAM J. Math. Anal.} 53(3), 3577-3601, 2021

 \bibitem{BDVV} S. Biagi,  S. Dipierro, E. Valdinoci, and  E. Vecchi.
 A Faber-Krahn inequality for mixed local and nonlocal operators,   \emph{J. Anal. Math.} 150, 405--448, 2023
 
\bibitem{BDVV2} S.  Biagi, S.  Dipierro, E.  Valdinoci, and  E. Vecchi.
Semilinear elliptic equations involving mixed local and nonlocal operators,
\textit{Proc. Roy. Soc. Edinburgh Sect. A, }151(5):1611--1641, 2021

\bibitem{BDVV3} S. Biagi,  S. Dipierro, E. Valdinoci, and  E. Vecchi.
 Mixed local and nonlocal elliptic operators: regularity and
maximum principles, \emph{Comm. Partial Differential Equations} 47, no. 3, 585--629, 2022

\bibitem{BDVV4} S. Biagi,  S. Dipierro, E. Valdinoci, and  E. Vecchi.
A Brezis-Nirenberg type result for mixed local and nonlocal operators,
\emph{ Nonlinear Differential Equations and Applications NoDEA } 32, no. 62, 2025

 \bibitem{BMS22} A. Biswas, M.  Modasiya, and A. Sen.
 Boundary regularity of mixed local-nonlocal operators and its applications,  \emph{Annali di Matematica Pura ed Applicata}, 202, 679--710, 2023 
 
 \bibitem{BM24} A. Biswas and M. Modasiya,
 Mixed local-nonlocal operators: maximum principles, eigenvalue problems and their applications,
 \emph{Journal d'Analysis Math.}, {\it to appear}, 2024
 
 
 \bibitem{BM03} Y. Bozhkov and E. Mitidieri.
 Existence of multiple solutions for quasilinear systems via fibering method,
 \emph{J. Diff. Equ.} 190(1), 239--267, 2003
 


\bibitem{DM24} C. De Filippis and G. Mingione.
Gradient regularity in mixed local and nonlocal problems, \emph{Math. Ann.} 388, 261--328, 2024

\bibitem{GK22} P. Garain and J. Kinnunen.
On the regularity theory for mixed local and nonlocal quasilinear
elliptic equations, \emph{Trans. Amer. Math. Soc.} 375, no. 8, 5393--5423, 2022

\bibitem{GilTru} D. Gilbarg and N. S. Trudinger.
Elliptic partial differential equations of second order, 
\emph{Second, Grundlehren der
Mathematischen Wissenschaften}, vol. 224, Springer-Verlag, Berlin, 1983.

  
\bibitem{MS23} M. Modasiya and A. Sen.
Fine boundary regularity for fully nonlinear mixed local-nonlocal problems,
\emph{arXiv: 2301.02397} , 2023
 
 \bibitem{MZ21} C. Mou and Y. Zhang.
 Regularity theory for second order integro-PDEs,
 \emph{Potential Anal.} 54, no.2, 387--407, 2021
 
 \bibitem{Poh} S. I. Pohozaev.
 On the eigenfunctions of the equation $\Delta u + \lambda f (u) = 0$, Dokl. Akad. Nauk
SSSR 165 , 1408--1411, 1965

\bibitem{Poh70} S.I. Pohozaev. On eigenfunctions of quasilinear elliptic problems, \emph{Mat. Sb.} 82,  192?212, 1970

\bibitem{PS86} P. Pucci and J. Serrin.
A general variational inequality, \emph{ndiana Univ. Math. J.} 35 , 681-703, 1986

\bibitem{RS14a}
X. Ros-Oton and J. Serra.
The Dirichlet problem for the fractional Laplacian: regularity up to the
boundary. 
\emph{J. Math. Pures Appl.} 101, 275--302, 2014

\bibitem{RS14} X.  Ros-Oton and J. Serra.
The Pohozaev identity for the fractional Laplacian, 
\emph{ Arch. Ration. Mech. Anal.} 213 , no. 2, 587--628, 2014
 
 \bibitem{RS15} X.  Ros-Oton and J. Serra.
 Nonexistence results for nonlocal equations with critical and supercritical nonlinearities,
 \emph{Comm. Partial Differential Equations} 40, 115--133, 2015
 
 \bibitem{RS15a} X.  Ros-Oton and J. Serra.
 Local integration by parts and Pohozaev identities for higher order fractional Laplacians,
 \emph{Discrete Contin. Dyn. Syst.} 35, no.5, 2131--2150, 2015
 
 \bibitem{RSV} X.  Ros-Oton and J. Serra and E. Valdinoci,
 Pohozaev identities for anisotropic integro-differential operators,
 \emph{Comm. Partial Differential Equations} 42, 1290--1321, 2017
 
\bibitem{Uh76} K. Uhlenbeck. Generic properties of eigenfunctions, Amer. J. Math. 98, 1059--1078, 1976

\end{thebibliography}

\end{document}